\documentclass[12pt]{amsart} 
\usepackage{amsmath,amssymb,amscd,amsthm}
\usepackage{latexsym}
\usepackage{graphicx}
\usepackage[english]{babel}
\usepackage[latin1]{inputenc}       
\usepackage{textcomp}
\usepackage{times}
\usepackage{pgf,pgfarrows,pgfnodes,pgfautomata,pgfheaps}
\usepackage{colortbl}

\newtheorem{thm}{Theorem}
\newtheorem{cor}[thm]{Corollary}
\newtheorem{theorem}{Theorem}[section]
\newtheorem{corollary}[theorem]{Corollary}
\newtheorem{proposition}[theorem]{Proposition}
\newtheorem{lemma}[theorem]{Lemma}

\newtheorem{remark}[theorem]{Remark}

\def\irr#1{{\rm Irr}(#1)}
\def\irrr#1#2 {\irr {#1 \mid #2}}
\newcommand{\R}{\mathbb R}

\newcommand{\sfe}{{{\mathbb S}^{n-1}}}

\newcommand{\K}{\mathcal K}
\newcommand{\C}{C^{2,+}(\sfe)}

\newcommand{\Sym}{{\rm Sym}}

\begin{document}

\title[Local log-Minkowski]{A note on the quantitative local version 
\\
of the log-Brunn-Minkowski inequality}
\author[Andrea Colesanti, Galyna Livshyts]{Andrea Colesanti, Galyna Livshyts}
\address{Dipartimento di Matematica e Informatica ``U. Dini",
Universit\`a degli Studi di Firenze}
\email{colesant@math.unifi.it}
\address{School of Mathematics, Georgia Institute of Technology} \email{glivshyts6@math.gatech.edu}

\subjclass[2010]{Primary: 52} 
\date{\today}
\begin{abstract} We prove that the log-Brunn-Minkowski inequality 
\begin{equation*}
|\lambda K+_0 (1-\lambda)L|\geq |K|^{\lambda}|L|^{1-\lambda}
\end{equation*}
(where $|\cdot|$ is the Lebesgue measure and $+_0$ is the so-called log-addition) holds when $K\subset\R^n$ is a ball and $L$ is a symmetric convex body in a suitable $C^2$ neighborhood of $K$.  
\end{abstract}
\maketitle

\section{Introduction}

The classical Brunn-Minkowski inequality (in its multiplicative form) states that for every pair of Borel-measurable sets $K$ and $L,$ and for a scalar 
$\lambda\in[0,1]$, one has
\begin{equation}\label{BM}
|\lambda K+(1-\lambda)L|\geq |K|^{\lambda} |L|^{1-\lambda}
\end{equation}
(in fact, we need to assume that $\lambda K+(1-\lambda)L$ is measurable as well).
See, {\em e.g.}, the extensive survey by Gardner \cite{Gar} on the subject.

Define the geometric average of convex bodies: 
$$
\lambda K+_0 (1-\lambda)L=\{x\in \R^n:\,\langle x,u\rangle \leq h_K^{\lambda}(u) h_L^{1-\lambda}(u),\,\forall u\in\sfe\},
$$
where $h_K$ and $h_L$ are the support functions of $K$ and $L$, respectively.
The log-Brunn-Minkowski conjecture (see Boroczky, Lutwak, Yang, Zhang \cite{BLYZ}) states that 
\begin{equation}\label{lbmc}
|\lambda K+_0 (1-\lambda)L|\geq |K|^{\lambda}|L|^{1-\lambda}
\end{equation}
for every pair of symmetric convex sets $K$ and $L$. 
Important applications and motivations for this conjecture can be found in \cite{BLYZ-1}, \cite{BLYZ-2}. In particular, it is 
equivalent to the famous B-conjecture (see \cite{B-conj}). Note that the straightforward inclusion
$$
\lambda K+_0 (1-\lambda)L\subset \lambda K+(1-\lambda)L
$$ 
tells us that (\ref{lbmc}) is stronger than the classical Brunn-Minkowski inequality (\ref{BM}).

It is not difficult to see that the condition of symmetry in \eqref{lbmc} is necessary. B\"or\"oczky, Lutwak, Yang and Zhang \cite{BLYZ} showed that this 
conjecture holds for $n=2$. Saroglou \cite{christos} proved that the conjecture is true when the 
sets $K$ and $L$ are unconditional ({\em i.e.} they are symmetric with respect to every coordinate hyperplane). Rotem \cite{liran} showed that 
log-Brunn-Minkowski conjecture holds for complex convex bodies. 

Recently, local versions of the log-Brunn-Minkowski inequality have been considered. In the paper \cite{CLM} the authors prove the following 
fact. {\em Let $R>0$ and let $\phi\in C^2(\sfe)$; then there exists $a>0$ such that if the support functions of $K$ and $L$ are 
$Re^{\epsilon_1\phi}$, $Re^{\epsilon_2\phi}$, respectively, with $0\le\epsilon_1,\epsilon_2<a$, then \eqref{lbmc} holds.} 
In this note we improve the previous result; in the next statement, which is our main result, $B_2^n$ denotes the unit ball in $\R^n$.

\begin{theorem}\label{main thm}
Let $R>0$ and $n\geq 2$. There exists $\epsilon(n)$ such that for every symmetric convex $C^2$-smooth body $K$ in $\R^n$ such that 
$\|h-R\|_{C^2(\sfe)}\le\epsilon(n)R$, where $h$ is the support function of $K$, we have
\begin{equation}\label{star}
|\lambda K+_0 (1-\lambda)(R B_2^n)|\geq |K|^{\lambda} |R B_2^n|^{1-\lambda}
\quad\forall\ \lambda\in[0,1].
\end{equation}
Moreover, equality holds if and only if $K$ is a ball centered at the origin.
\end{theorem}

In other words, Theorem \ref{main thm} asserts that if $K$ is a ball, there exists a $C^2$ neighborhood $\mathcal{U}$ of $K$, 
whose radius depends on the dimension $n$ only, such that \eqref{lbmc} holds for $K$ and for every $L\in\mathcal{U}$.

We remark, that our result is contained in the recent paper by Kolesnikov and Milman \cite{KolMilsupernew} (see  Theorem 1.2), which was unknown to us as of this writing.

\medskip

The cone volume measure of a convex set $K$ is the measure on the sphere, defined as
$$c_K(\Omega)=\frac{1}{n}\int_{\Omega} h_K(u) d s_K(u),$$
where $h_K$ is the support function of $K.$
It was conjectured by Lutwak \cite{lutwak} that the cone volume measure determines a symmetric convex body uniquely. As a corollary of our main result, we deduce a local uniqueness result:
\begin{cor}\label{maincor}
Let $n\geq 2$ and let $R>0$ be a constant. There exists $\epsilon=\epsilon(n)>0$, which depends only on the dimension, such that, given a symmetric $C^2$-smooth convex body $K$ satisfying $\|R-h_K\|_{C^2(\sfe)}\le\epsilon(n) R$, and $dc_K(u)=R^n du,$ one has that $K$ coincides with the Euclidean ball of radius $R$.
\end{cor}
The previous corollary can be also deduced by the more general results proved by Firey in \cite{Firey} (see in particular Theorem 3 therein). However, we believe that our methods present separate interest.
\medskip

\noindent
\textbf{Aknowledgement.} The authors would like to thank Alina Stancu and Mohammad Ivaki for precious conversations on the subject of this paper. 
The first author is supported
by the G.N.A.M.P.A., project {\em Problemi sovradeterminati e questioni di ottimizzazione di forma}.
The second author is supported by the NSF CAREER DMS-1753260. The work was partially supported by the National Science Foundation under Grant No. DMS-1440140 while the second author was in residence at the Mathematical Sciences Research Institute in Berkeley, California, during the Fall 2017 semester.

\section{Preliminaries}

We work in the Euclidean $n$-dimensional space $\R^n$. The unit ball shall be denoted by $B_2^n$ and the unit sphere by $\sfe$. 
The Lebesgue volume of a measurable set $A\subset \R^n$ is denoted by $|A|$.

We say that a convex body $K$ is of class $C^{2,+}$ if $\partial K$ is of class $C^2$ and the Gauss curvature is strictly positive at 
every $x\in\partial K$. 
In particular, if $K$ is $C^{2,+}$ then it admits unique outer unit normal $\nu_K(x)$ at every boundary point $x$. Recall that the Gauss map $\nu_K\,:\,\partial K\to\sfe$ is the map assigning the collection of unit normals to each point of $\partial K.$ 

We recall that an orthonormal frame on the sphere is a map which associates to every $x\in\sfe$ an orthonormal basis of the tangent space
to $\sfe$ at $x$.
Let $\psi\in C^2(\sfe)$; we denote by $\psi_i(u)$ and $\psi_{ij}(u)$, $i,j\in\{1,\dots,n-1\}$, 
the first and second covariant derivatives of $\psi$ at $u\in\sfe$, with respect to a fixed local orthonormal frame on an 
open subset of $\sfe$. We define the matrix
\begin{equation}\label{curvature_matrix}
Q(\psi;u)=(q_{ij})_{i,j=1,\dots,n-1}=\left(
\psi_{ij}(u)+\psi(u)\delta_{ij}
\right)_{i,j=1,\dots,n-1},
\end{equation}
where the $\delta_{ij}$'s are the usual Kronecker symbols. On an occasion, instead of $Q(\psi;u)$ we write $Q(\psi)$. Note that $Q(\psi;u)$ is symmetric by standard properties of covariant derivatives. In what follows we shall often consider $\psi$ to be a support function of a convex body $K$. In this case $Q(\psi)$ is called {\em curvature matrix} of $K$; this name comes from the fact that $\det(Q(\psi))$ is the density of the curvature measure $s_K$, and therefore, 
$$
|K|=\frac{1}{n}\int_{\sfe} h_K(u) \det Q(h_K,u) du.
$$
(See, for instance, Koldobsky \cite{Kold} for the proof.) We recall here a fact that will be frequently used in the paper 
(a proof can be deduced, for instance, from \cite[Section 2.5]{book4}).

\begin{proposition}\label{added} 
Let $K\in\K^n$ and let $h$ be its support function. Then $K$ is of class $C^{2,+}$ if and only if
$h\in C^2(\sfe)$ and
$$
Q(h;u)>0,\quad\forall\ u\in\sfe.
$$ 
\end{proposition}

In view of the previous result, we say that a function $h$ defined on $\sfe$ is of class $C^{2,+}(\sfe)$ if $h\in\C^(\sfe)$ and 
$Q(h;u)$ is positive definite for every $x\in\sfe$. For $g\in C^2(\sfe)$, we set
$$
\| g\|_{C^2(\sfe)}=
\|g\|_{L^\infty(\sfe)}+\|\nabla_s g\|_{L^\infty(\sfe)}+\sum_{i,j}^{n-1} \|g_{ij}\|_{L^\infty(\sfe)},
$$
where $\nabla_s g$ denotes the spherical gradient of $g$ ({\em i.e.} the vector having first covariant derivatives as components). 
We also set
$$
\|g\|_{L^2(\sfe)}^2=\int_{\sfe} g^2(u)du,\quad
\|\nabla_s g\|^2_{L^2(\sfe)}=\int_{\sfe}\|\nabla_s g(u)\|^2du.
$$

\subsection{Co-factor matrices}
For a natural number $N$, denote by $\Sym(N)$ the space of $N\times N$ symmetric matrices. Given $A\in\Sym(N)$
we denote by $a_{jk}$ its $jk$-th entry and write $A=(a_{jk})$. For $j,k=1,\dots,N$ we set
\begin{equation}\label{c_ij}
c_{jk}(A)=\frac{\partial\det}{\partial a_{jk}}(A).
\end{equation}
The matrix $(c_{jk}(A))$ is called the co-factor matrix of $A$. We also set, for $j,k,r,s=1\dots,N$,
\begin{equation}\label{c_ijkl}
c_{jk,rs}(A)=\frac{\partial^2\det}{\partial a_{jk}\partial a_{rs}}(A).
\end{equation}

Recall that
\begin{equation}\label{determinant}
\det(A)=\frac{1}{N!}\sum{\delta\binom{j_1,\dots,j_{N}}{k_1,\dots,k_{N}}a_{j_1
k_1}\cdots a_{j_{N} k_{N}}},
\end{equation}
where the sum is taken over all possible indices $j_s, k_s \in
\{1,\dots,N\}$ (for $s=1,\dots,N$) and the Kronecker symbol
$$
\delta\binom{j_1,\dots,j_{N}}{k_1,\dots,k_{N}}
$$ 
equals $1$ (respectively, $-1$) when $j_1,\dots,j_{N}$ are distinct and
$(k_1,\dots,k_{N})$ is an even (respectively, odd) permutation
of $(j_1,\dots,j_{N})$; otherwise it is $0$. Using (\ref{determinant}), along with (\ref{c_ij}) and (\ref{c_ijkl}), we derive for every $j,k,r,s\in\{1,\dots,N\}$:
\begin{eqnarray}\label{refer}
c_{jk}(A)&=&\frac{1}{(N-1)!}\sum{\delta\binom{j,j_1,\dots,j_{N-1}}{k,j_1,\dots,k_{N-1}}a_{j_1
k_1}\cdots a_{j_{N-1} k_{N-1}}}\,,\nonumber\\
\nonumber\\
c_{jk,rs}(A)&=&\frac{1}{(N-2)!}\sum{\delta\binom{r,j,j_1,\dots,j_{N-2}}{s,k,k_1,\dots,k_{N-2}}a_{j_1
k_1}\cdots a_{j_{N-2} k_{N-2}}}\,.
\end{eqnarray}

\begin{remark}\label{remark matrices 1} If $A\in\Sym(N)$ is invertible, then, by (\ref{refer}),
$$
(c_{jk}(A))=\det(A)\ A^{-1}.
$$ 
In particular, if $A=I_N$ (the identity matrix of order $N$), then $(c_{jk})(I_N)=I_N.$ 
\end{remark}

\begin{remark}\label{remark matrices 2} Observe that, by (\ref{refer}), for every $A=(a_{jk})\in\Sym(N)$,
$$
\sum_{j,k=1}^N c_{jk}(A)a_{jk}=N\det(A).
$$

\end{remark}

\begin{remark} Let $A=(a_{ij})\in\Sym(N)$ and let $M>0$ be such that
$$
|a_{jk}|\le C,\quad\forall\, j,k=1\dots,M.
$$
Then there exists some constant $c=c(N)$ ({\em i.e.} depending only on $N$) such that, for every $j,k,r,s=1,\dots, N$,
$$
|c_{jk}(A)|\le c(N)\ M^{N-1},\quad
|c_{jk,rs}(A)|\le c(N)\ M^{N-2}.
$$

\end{remark}
Note that if $g\equiv c$ on $\sfe$ then $Q(g;u)=cI_{n-1}$ for every $u\in\sfe$.

\subsection{The Cheng-Yau lemma and an extension}

Let $h\in C^3(\sfe)$. Consider the co-factor matrix $y\to C[Q(h;y)]$. This is a matrix of functions on $\sfe$. The lemma of Cheng and Yau (\cite{Cheng-Yau}) asserts
that each column of this matrix is divergence-free.

\begin{lemma}[Cheng-Yau] Let $h\in C^3(\sfe)$. Then, for every index $j\in\{1,\dots,n-1\}$ and for every $y\in\sfe$,
$$
\sum_{i=1}^{n-1}\left(
c_{ij}[Q(h;y)]
\right)_i=0,
$$
where the sub-script $i$ denotes the derivative with respect to the $i$-th element of an orthonormal frame on $\sfe$. 
\end{lemma}

For simplicity of notation we shall often write $C(h)$, $c_{ij}(h)$ and $c_{ij,kl}(h)$ in place of $C[Q(h)]$, $c_{ij}[Q(h)]$ and $c_{ij,kl}[Q(h)]$ respectively. 
As a corollary of the previous result we have the following integration by parts formula. 
If $h,\psi,\phi\in C^2(\sfe)$, then
\begin{equation}\label{ibp1}
\int_\sfe \phi\,c_{ij}(h)(\psi_{ij}+\psi\,\delta_{ij})dy
=\int_\sfe \psi\,c_{ij}(h)(\phi_{ij}+\phi\,\delta_{ij})dy.
\end{equation}

The Lemma of Cheng and Yau admits the following extension (see Lemma 2.3 in \cite{Colesanti-Hug-Saorin2}). 
Note that we adopt the summation convention over repeated indices. 

\begin{lemma} Let $h,\psi\in C^3(\sfe)$. Then, for every $k\in\{1,\dots,n-1\}$ and for every $y\in\sfe$
$$
\sum_{l=1}^{n-1}\left(
c_{ij,kl}[Q(h;y)](\psi_{ij}+\psi\delta_{ij})
\right)_l=0.
$$
\end{lemma}

Correspondingly we have, for every $h,\psi,\varphi,\phi\in C^2(\sfe)$,
\begin{eqnarray}\label{ibp2}
&&\int_\sfe \psi\,c_{ij,kl}(h)(\varphi_{ij}+\varphi\delta_{ij})((\phi)_{kl}+\phi\,\delta_{kl})dy\nonumber\\
&&=\int_\sfe \phi\,c_{ij,kl}(h)(\varphi_{ij}+\varphi\delta_{ij})((\psi)_{kl}+\psi\,\delta_{kl})dy.
\end{eqnarray}

\subsection{A Poincar\'e inequality for even functions on the sphere}

Here we use some basic facts from the theory of spherical harmonics,
which can be found, for instance in \cite{Groemer,Kold} or in \cite[Appendix]{book4}. We denote by $\Delta_\sigma$ the spherical Laplace operator (or Laplace-Beltrami operator), on $\sfe$. 
The first eigenvalue of $\Delta_\sigma$ is $0$, and the corresponding eigenspace is formed by constant functions. The second eigenvalue of $\Delta_\sigma$ is $n-1$, and the corresponding eigenspace is formed by the restrictions of linear functions of $\R^n$ to $\sfe$. The third eigenvalue is $2n$, which implies, in particular, that for any \textbf{even} function $\psi\in C^2(\sfe)$ such that
$$
\int_{\sfe}\psi du=0,
$$
one has
\begin{equation}\label{spherical harmonics}
\int_{\sfe}\psi^2du\le\frac 1{2n}\int_{\sfe}|\nabla_s\psi|^2du.
\end{equation}

\section{Computations of derivatives}
Let $\psi\in C^2(\sfe),$ and let $s>0$. We consider the function $h_s(u)=e^{s\psi(u)}$. We will denote derivatives with respect to the parameter $s$ by a dot, {\em e.g.}:
$$
\dot{h}_s(u)=\frac{d}{ds}h(u),\quad\ddot{h}_s(u)=\frac{d^2}{ds^2}h(u),\ \dots
$$
Note that
\begin{equation}\label{dots}
\dot h_s=\psi h_s,\quad\ddot h_s=\psi^2 h_s,\quad \dddot h_s=\psi^3 h_s.
\end{equation}

\begin{remark}\label{rem_change} 
As we may interchange the order of derivatives, for every $j,k=1,\dots,n-1$ we have
$$
q_{jk}(\dot{h})=\dot{q}_{jk}(h),
$$
and thus
$$
\dot{Q}(h)=Q(\dot{h}). 
$$
Similar equalities hold for successive derivatives in $s$.
\end{remark}
Consider the volume function
\begin{equation}\label{the function}
f(s)=\frac{1}{n}\int_\sfe h_s(u)\det(Q(h_s; u))du.
\end{equation}
If $h_s$ is the support function of a convex body $K_s$ (as it will be in the sequel), $f$ represents the volume of $K_s$.  

\begin{remark}\label{remark large neighborhood}
The entries of $Q(h_s)$ are continuous functions of the second derivatives of $h_s$ and 
$Q(h_0)>0$. Hence there exists $\eta_0>0$ such that if $\psi\in C^2(\sfe)$ is such that
$\|\psi\|_{C^2(\sfe)}\le\eta_0$, then 
\begin{equation}\label{large neighborhood}
Q(e^{s\psi};u)>0\quad\forall\ u\in\sfe,\ \forall\ s\in[-2,2].
\end{equation}
We shall use notation
$$
\mathcal U=\{\psi\in C^2(\sfe)\colon \|\psi\|_{C^2(\sfe)}\le\eta_0\}.
$$
Note that if $\psi\in{\mathcal U}$ then $f>0$ in $[-2,2]$. Moreover, in the case $h_0\equiv 1$ we have $Q(h_0)=I_{n-1}$, and
\begin{equation}\label{function in zero}
f(0)=\frac1n|\sfe|.
\end{equation}
\end{remark}

\begin{lemma}\label{lemma derivatives 1} In the notations introduced above, we have, for every $s$:
\begin{equation}\label{first derivative}
f'(s)=\int_{\sfe} \psi h_s\det(Q(h_s))du;
\end{equation}
\begin{equation}\label{second derivative}
f''(s)=\int_\sfe\left[\psi^2{h}_s\det(Q(h_s))+\psi h_sc_{jk}(h_s) q_{jk}(\psi h_s)\right]du;
\end{equation}
\begin{eqnarray}\label{third derivative}
f'''(s)&=&\int_{\sfe}h_s\left[\psi^3\det(Q(h_s))+2\psi^2c_{jk}(h_s)q_{jk}(\psi h_s)\right]du+\\
&+&\int_\sfe h_s \left\{\psi \left[c_{jk,rs}(h_s)q_{jk}(\psi h_s)q_{rs}(\psi h_s)+c_{jk}(h_s)q_{jk}(\psi^2 h_s)\right]\right\}du.
\nonumber
\end{eqnarray}

\end{lemma}

\begin{proof}
For brevity, we set
$$
c_{jk}(h)=c_{jk}(Q(h)).
$$
We differentiate the function $f$ in $s$, and we adopt the summation convention over repeated indices.
\begin{eqnarray*}
f'(s)&=&\frac 1n\int_\sfe\left[\dot{h}\det(Q(h))+h c_{jk}(h)\dot{q}_{jk}(h)\right]dy\\
&=&\frac 1n\int_\sfe\left[\dot{h}\det(Q(h))+h c_{jk}(h)q_{jk}(\dot{h})\right]dy\\
&=&\frac 1n\int_\sfe\left[\dot{h}\det(Q(h))+\dot{h} c_{jk}(h) q_{jk}(h)\right]dy\\
&=&\int_\sfe \dot{h}\det(Q(h))dy.
\end{eqnarray*}
Above we have used Remark \ref{rem_change} and the integration by parts formula (\ref{ibp1}).

Passing to the second derivative, we get:
\begin{eqnarray*}
f''(s)&=&\int_\sfe\left[\ddot{h}\det(Q(h))+\dot{h}c_{jk}(h)\dot{q}_{jk}(h)\right]du\\
&=&\int_\sfe\left[\ddot{h}\det(Q(h))+\dot{h}c_{jk}(h) q_{jk}(\dot{h})\right]du.
\end{eqnarray*}

Finally
\begin{eqnarray*}
f'''(s)&=&\int_{\sfe}\left[\dddot{h}\det(Q(h))+2\ddot{h}c_{jk}(h)q_{jk}(\dot{h})\right]du+\\
&&+\int_\sfe\left\{\dot{h}\left[c_{jk,rs}(h)q_{jk}(\dot{h})q_{rs}(\dot{h})+c_{jk}q_{jk}(\ddot{h})\right]\right\}du.
\end{eqnarray*}
Equalities \eqref{first derivative}, \eqref{second derivative} and \eqref{third derivative} follow from \eqref{dots}.

\end{proof}
The next Corollary has appeared in \cite{CLM}.

\begin{corollary}\label{lemma derivatives 2} In the notations introduced before we have:
\begin{equation}\label{first derivative at zero}
f'(0)=\int_{\sfe}\psi du;
\end{equation}
\begin{equation}\label{second derivative at zero}
f''(0)=\int_{\sfe}[n\psi^2-|\nabla_s\psi|^2]du.
\end{equation}
\end{corollary}

\begin{proof} Equality \eqref{first derivative at zero} follows immediately from \eqref{first derivative}.
Moreover, plugging $s=0$ in \eqref{second derivative}, and using the facts
$$
c_{jk}(h_0)=\delta_{jk}\quad\mbox{and}\quad
q_{jk}(\psi)=(\psi_{jk}+\psi\delta_{kj})
$$
for every $j,k=1,\dots,n-1$, we get
$$
f''(0)=\int_{\sfe}[n\psi^2+\psi\Delta_s\psi]du.
$$
By the divergence theorem on $\sfe$ we deduce \eqref{second derivative at zero}.

\end{proof}

\begin{lemma}\label{estimate of third derivative}
For every $\rho>0$ there exists $\eta>0$, such that if $\psi\in{\mathcal U}$ is an even function and it verifies:
\begin{itemize}
\item
$$
\int_{\sfe}\psi du=0;
$$
\item
$$
\|\psi\|_{C^2(\sfe)}\le\eta;
$$
\end{itemize}
then 
$$
|(\log f)'''(s)|\le\rho\|\nabla_s\psi\|^2_{L^2(\sfe)},\quad\forall\ s\in[-2,2],
$$
where $f$ is defined as in \eqref{the function} and $h_s=e^{s\psi}$.
\end{lemma}

\begin{proof} We have 
\begin{eqnarray*}
(\log f)'''(s)=\frac{f'''(s)}{f(s)}-3\frac{f'(s)f''(s)}{f^2(s)}+2\frac{(f')^3(s)}{f^3(s)}.
\end{eqnarray*}
we first fix $\eta_1>0$ such that $\|\psi\|_{C^2(\sfe)}\le\eta_1$ implies 
$$
f(s)\ge\frac1
{4n}|\sfe|=\frac 14f(0),\quad\forall\ s\in[-2,2].
$$
Hence
$$
|(\log f)'''(s)|\le C_0(|f'''(s)|+|f'(s)f''(s)|+|(f')^3(s)|)=C_0(T_1+T_2+T_3),
$$
for some constant $C_0=C_0(n)=\frac{4n}{|\sfe|}$. Throughout this proof, we will denote by $C$ a generic positive constant dependent 
on the dimension $n$ and $\eta_1$.

\medskip 

\noindent{\bf Bound of the term $T_3$.}
There exists $C$ such that
$$
\|h_s\|_{C^2(\sfe)}=\|e^{s\psi}\|_{C^2(\sfe)}\le C,
$$
for every $s\in[-2,2]$ and for every $\psi\in{\mathcal U}$. Therefore
$$
h_s(u)\det(Q(h_s;u))\le C,\quad\forall\ \psi\in\mathcal{U}.
$$
Consequently, by Lemma \ref{lemma derivatives 1}, we may write two types of estimates
\begin{eqnarray*}\label{derest}
|f'(s)|\le C\|\psi\|_{C^2(\sfe)},\quad
|f'(s)|\le C\|\psi\|_{L_2(\sfe)}.
\end{eqnarray*}
By (\ref{spherical harmonics}), there exists $\eta'>0$ such that 
\begin{equation}\label{bound}
|(f')^3(s)|\le\frac\rho{3C_0}\|\nabla \psi\|_{L^2(\sfe)},
\end{equation}
if $\psi$ verifies $\|\psi\|^2_{C^2(\sfe)}\le \eta'$.

\medskip

\noindent{\bf Bound of the term $T_2$.} 
By Lemma \ref{lemma derivatives 1}, (\ref{spherical harmonics}) and the integration by parts formula (\ref{ibp1}), we have
\begin{eqnarray*}
|f''(s)|&\le& C\|\psi\|^2_{L^2(\sfe)}+\left|\int_{\sfe}\psi h_sc_{jk}(h_s)(\psi h_s\delta_{jk}+(\psi h_s)_{jk})du\right|\\
&\le&C\|\psi\|^2_{L^2(\sfe)}+\left|\int_{\sfe}c_{jk}(h_s)(\psi h_s)_j(\psi h_s)_{k}du\right|\\
&\le&C\|\psi\|^2_{L^2(\sfe)}+C\|\nabla_s\psi\|_{L^2(\sfe)}^2\\
&\le&C\|\nabla_s\psi\|_{L^2(\sfe)}^2
\end{eqnarray*}
(note that the first term was bounded using the argument as for the previous part of this proof). 
Hence we have the bound \eqref{bound} for $T_2$ as well.

\medskip

\noindent{\bf Bound of the term $T_1$.} Equality \eqref{third derivative} provides an expression of $f'''(s)$ as the sum of four terms.
Each of them can be treated as in the previous two cases, with the exception of $\left|\int_{\sfe}\psi h_s c_{jk,rs}(h_s)q_{rs}(\psi h_s)q_{jk}(\psi h_s)du\right|$. We estimate it as follows:
\begin{eqnarray*}
&&\left|\int_{\sfe}\psi h_s c_{jk,rs}(h_s)q_{rs}(\psi h_s)q_{jk}(\psi h_s)du\right|\\
&\le&\left|\int_{\sfe}\psi^2 h_s^2 c_{jk,rs}(h_s)q_{rs}(\psi h_s)\delta_{jk}du\right|+
\left|\int_{\sfe}\psi h_s c_{jk,rs}(h_s)q_{rs}(\psi h_s)(\psi h_s)_{jk}du\right|\\
&\le&C\|\psi\|_{C^2(\sfe)}\,\|\psi\|^2_{L_2(\sfe)}+
\left|\int_{\sfe} c_{jk,rs}(h_s)q_{rs}(\psi h_s)(\psi h_s)_{j}(\psi h_s)_{k}du\right|\\
&\le&C\|\psi\|_{C^2(\sfe)}\,\|\psi\|^2_{L_2(\sfe)}+C\|\psi\|_{C^2(\sfe)}\,\|\nabla_s \psi\|^2_{L_2(\sfe)}\\
&\le&C\|\psi\|_{C^2(\sfe)}\,\|\nabla_s \psi\|^2_{L_2(\sfe)}.
\end{eqnarray*}

We deduce that the upper bound \eqref{bound} can be established for $T_1$. This concludes the proof.

\end{proof}

\begin{lemma}\label{lemma concavity} Let $f$ be defined by \eqref{the function}. There exists $\eta>0$ such that for every even $\psi\in{\mathcal U}$ so that $\|\psi\|_{C^2(\sfe)}\le\eta$, the function $\log(f(s))$, 
is concave in $[-2,2]$. Moreover it is strictly concave in this interval unless $\psi$ is constant. 
\end{lemma}

\begin{proof}
We first assume that
\begin{equation}\label{zero mean}
\int_{\sfe} \psi du=0.
\end{equation}
For every $s\in[-2,2]$ there exists $\bar s$ between $0$ and $s$ such that
$$
\left(\log f\right)''(s)=(\log f)''(0)+s(\log f)'''(\bar s)=\frac{f(0)f''(0)-f'(0)^2}{f(0)^2}+s(\log f)'''(\bar s).
$$
It is shown in Lemma \ref{estimate of third derivative} that, for an arbitrary $\rho>0$ there exists
$\eta>0$ such that if $\|\psi\|_{C^2(\sfe)}\le\eta$ then 
$$
(\log f)'''(s) \leq \rho\|\nabla_s\psi\|^2_{L^2(\sfe)},\quad\forall\ s\in[-2,2].
$$
Using the last inequality along with Lemma \ref{lemma derivatives 2} and \eqref{zero mean} we have 
$$
\left(\log f\right)''(s)\leq\frac{1}{|\sfe|}\left[\int_{\sfe} (n\psi^2(u)-|\nabla_s \psi(u)|^2)du\right]+
\rho||\nabla_s\psi||^2_{L^2}.
$$
By \eqref{spherical harmonics} we deduce
$$
\left(\log f\right)''(s)\leq \|\nabla_s\psi\|^2_{L^2(\sfe)}\left(
2\rho-\frac1{2|\sfe|}\right),
$$
which is negative as long as
$$
\rho<\frac1{4|\sfe|},
$$
and, with this choice, strictly negative unless $\psi$ is a constant function.

Next we drop the assumption \eqref{zero mean}. For $\psi\in C^2(\sfe)$, let
$$
m_\psi=\frac1{|\sfe|}\int_{\sfe}\psi du,\quad\mbox{and}\quad\bar\psi=\psi-m_\psi.
$$
Clearly $\bar\psi\in C^2(\sfe)$ and $\bar\psi$ verifies condition \eqref{zero mean}. Moreover,
$$
\|\bar\psi\|_{C^2(\sfe)}\le\|\psi\|_{C^2(\sfe)}+|m_\psi|\le 2\|\psi\|_{C^2(\sfe)}.
$$
Consequently, $\bar\psi\in{\mathcal U}$ if $\|\psi\|_{C^2(\sfe)}\le\eta_0/2$. We also have:
$$
\bar h_s:=e^{s\bar\psi}= e^{s(\psi-m_\psi)}=e^{-sm_\psi}\ h_s.
$$
Hence 
$$
Q(\bar h_s)=e^{-sm_\psi}Q(h_s).
$$
Consider
$$
\bar f(s):=\frac 1n\int_{\sfe}\bar h_s\det(Q(\bar h_s))du=
e^{-ns m_\psi} f(s).
$$
We observe that $\log(\bar f)$ and $\log(f)$ differ by a linear term and convexity (resp. strict convexity) of $f$ is 
equivalent to convexity (resp. strict convexity) of $\bar f$. On the other hand, by the first part of this proof 
$\log(\bar f)$ is concave as long as $\|\bar\psi\|_{C^2(\sfe)}$ is sufficiently small, and this condition is verified
when, in turn, $\|\psi\|_{C^2(\sfe)}$ is sufficiently small. The proof is concluded.

\end{proof}

\section{Proofs.}

\subsection{Proof of Theorem \ref{main thm}.}

We assume $R=1$; the general case can be deduced by a scaling argument.

We first suppose that $\|h-1\|_{C^2(\sfe)}\le 1/4$. 
This implies that $h>0$ on $\sfe$ and therefore we may write $h$ in the form 
$h=e^{\psi}$, where $\psi=\log(h)\in C^2(\sfe)$.  

We select $\epsilon_0>0$ such that $\|h-1\|_{C^2(\sfe)}\le\epsilon_0$ implies $\|\psi\|_{C^(\sfe)}\le\eta_0$, {\em i.e.} 
$\psi\in{\mathcal U}$ (see Remark \ref{remark large neighborhood}). As a consequence of Proposition \ref{added}, 
$h_s=e^{s\psi}$ is the support
function of a $C^{2,+}$ convex body, for every $s\in[-2,2]$. In particular, for every $\lambda\in[0,1]$, the function $e^{\lambda\psi}$
is the support function of
$$
K^\lambda(B^n_2)^{1-\lambda}.
$$

There exists $\epsilon>0$ such that $\|h-1\|_{C^2(\sfe)}\le\epsilon$ implies $\|\psi\|_{C^2(\sfe)}\le\eta$, where $\eta>0$ is 
the quantity indicated in Lemma \ref{lemma concavity}. By the conclusion of Lemma \ref{lemma concavity}, the function $f(\lambda)=|K^\lambda(B^n_2)^{1-\lambda}|$ is log-concave, and hence (\ref{star}) follows. The equality case follows from the fact that the log-concavity of $f$ is strict unless $\psi$ is a constant function, which corresponds to the case when $K$ is a ball. 
\flushright{$\square$}
\flushleft

Below we shall sketch the proof of the corollary; we shall follow essentially the same scheme as in \cite{BLYZ}.

\subsection{Sketch of the proof of the Corollary \ref{maincor}.}

Firstly, by integrating the condition $d c_K(u)=R^{n} du$ over the sphere, we get $|K|=|R B_2^n|.$ Theorem \ref{main thm} implies (see \cite{BLYZ}): 
$$\int_{\sfe} \log \frac{R}{h_K} d c_K(u)\geq \log \frac{|R B_2^n|}{|K|}=0,$$
or, equivalently,
$$\int_{\sfe} \log R d c_K(u)\geq \int_{\sfe} \log h_K d c_K(u).$$
Using the fact that $dc_K(u)=R^n du$ once again, and then applying Theorem \ref{main thm} again, we see that the right hand side of the above is equal to
$$\int_{\sfe} \log h_K d_{R B_2^n}(u)\geq \int_{\sfe} \log R d_{R B_2^n}(u).$$
Note that the above is equal to
$$\int_{\sfe} \log R d c_K(u).$$
We have obtained a chain of inequalities starting and ending with the same expression, and hence equality must hold in all the inequalities. Therefore, $K$ is a Euclidean ball. Since, in addition, $|K|=|R B_2^n|$, we see that $K=R B_2^n,$ which finishes the proof. 
\flushright{$\square$}
\flushleft

\end{document}